\documentclass[10pt]{article}
 \makeatletter
\newfont{\footsc}{cmcsc10 at 8truept}
\newfont{\footbf}{cmbx10 at 8truept}
\newfont{\footrm}{cmr10 at 10truept}
\renewcommand{\ps@plain}{%
\renewcommand{\@oddfoot}{\footsc {\footbf }  \footrm\thepage}}

\makeatother  \leftmargin=25mm
\usepackage{color}
\usepackage[centertags]{amsmath}
\usepackage{tikz}
\usepackage{pgfplots}
\usepackage{amsfonts}
\usepackage{amsthm}
\usepackage{newlfont}
\usepackage{graphicx}
\usepackage{epstopdf}
\usepackage{subfigure}
\usepackage[utf8]{inputenc}
\usepackage{indentfirst}
\usepackage{latexsym}
\usepackage{float}
\usepackage{cite}
\usepackage{CJK}
\usepackage{array}
\usepackage{amsfonts}
\usepackage{mathrsfs}
\usepackage{amssymb}
\usepackage{mathtools}
\usepackage[ruled,linesnumbered]{algorithm2e}
\usepackage{float}     
\usepackage[T1]{fontenc}
\usepackage{authblk}
\usepackage{setspace}

\usepackage[colorlinks,
urlcolor=blue,
linkcolor=black,
anchorcolor=black,
citecolor=black]{hyperref}

\newtheorem{thm}{Theorem}[section]
\newtheorem{cor}[thm]{Corollary}
\newtheorem{lem}[thm]{Lemma}
\newtheorem{prop}[thm]{Proposition}
\newtheorem{prob}{Problem}
\newtheorem{conj}{Conjecture}
\theoremstyle{definition}
\newtheorem{defn}[thm]{Definition}
\newtheorem{case}{Case}
\theoremstyle{remark}
\newtheorem{rem}[thm]{Remark}

\title{On the classification of lattice polytopes via affine equivalence\footnotemark[2]\footnotetext[2]{This work is supported by the National Natural Science Foundation of China (NSFC12226006, NSFC11921001), the National Key
Research and Development Program of China (2018-YFA0704701) and China Scholarship Council.}}
\date{}
\author{Zhanyuan Cai\textsuperscript{1}, Yuqin Zhang\textsuperscript{2} and Qiuyue Liu\textsuperscript{1,}\thanks{Corresponding author. E-mail: 17864309562@163.com} \\  
		{\small $^{1}$Center for Applied Mathematics, Tianjin University, 300072, Tianjin, China}\\
		{\small $^{2}$School of Mathematics, Tianjin University, Tianjin, 300072, China}\\
  }
\begin{document}
	\maketitle
\begin{abstract}
In 1980, V. I. Arnold studied the classification problem for convex lattice polygons of a given area. Since then, this problem and its analogues have been studied by many authors, including Bárány, Lagarias, Pach, Santos, Ziegler and Zong. Despite extensive study, the structure of the representative sets in the classifications remains unclear, indicating a need for refined classification methods. In this paper, we propose a novel classification framework based on affine equivalence, which offers a fresh perspective on the problem. Our approach yields several classification results that extend and complement Bárány's work on volume and Zong's work on cardinality. These new results provide a more nuanced understanding of the structure of the representative set, offering deeper insights into the classification problem.\\[3mm]
\textbf{Keywords}:\  Lattice polytopes, classification, unimodular equivalence, affine equivalence
\\[3mm]
\textbf{2010 MSC}:\  52B20, 52C07
\end{abstract}

\maketitle
\section{Introduction}
Convex lattice polytopes in $\mathbb{E}^{d}$ are convex hulls of finite subsets of the integer lattice $\mathbb{Z}^{d}$. As usual, let $P$ denote a $d$-dimensional convex lattice polytope, $|P|$ the cardinality of $P$, $f_{0}(P)$ the number of vertices of $P$, and $\mathrm{vol}(P)$ the volume of $P$. For further details on polytopes and lattice polytopes, we refer to  \cite{Ziegler} and \cite{Gruber}.

This paper focuses on two primary concepts: unimodular equivalence and affine equivalence. Two $d$-dimensional convex lattice polytopes $P_{1}$ and $P_{2}$ are said to be \textit{unimodular equivalent} if there exists a unimodular transformation $\sigma$, i.e. an affine transformation with determinant $\pm1$ and preserving $\mathbb{Z}^{d}$, such that $\sigma(P_{1})=P_{2}$. It is straightforward to see that if $P_1$ is unimodular equivalent to $P_2$, then
\begin{equation*}
    |P_1|=|P_2|,\ f_0(P_1)=f_0(P_2)\ \text{and}\ \mathrm{vol}(P_1)=\mathrm{vol}(P_2).
\end{equation*}

Alternatively, we say $P_1$ and $P_2$ are \textit{affinely equivalent} if there exists an affine transformation $\sigma'$ such that $\sigma'(P_1)=P_2$. Both unimodular equivalence and affine equivalence are equivalence relations, meaning convex lattice polytopes can be categorized into different classes based on these relations. However, a significant distinction between them is that many invariants under unimodular equivalence, such as cardinality and volume, are not necessarily preserved under affine equivalence. This difference will lead to challenges in verifying affine equivalence.

By triangulations, it can be shown that
\begin{equation*}
    d!\cdot v(P)\in \mathbb{Z}
\end{equation*}
for any $d$-dimensional convex lattice polytope $P$. Let $v(d,m)$ denote the number of distinct classes of $d$-dimensional convex lattice polytopes with volume at most $m/d!$, where both $d$ and $m$ are positive integers. Let $f(d,m)$ and $g(d,m)$ be functions of positive integers $d$ and $m$. In this paper, $f(d,m)\ll g(d,m)$ means that
\begin{equation*}
    f(d,m)\leqslant c_{d}\cdot g(d,m)
\end{equation*}
holds for all positive integers $m$ and a fixed positive integer $d$, where $c_{d}$ is a suitable constant that depends only on $d$.

Next, we introduce Arnold's question and summarize research related to the classification of lattice polytopes. In 1980, V. I. Arnold \cite{arnol1980statistics} studied the values of $v(2,m)$ and proved that
\begin{align*}
    m^{\frac{1}{3}}\ll \log v(2,m)\ll  m^{\frac{1}{3}}\log m,
\end{align*}
holds for sufficiently large $m$.  In 1991, Lagarias and Ziegler \cite{Lagarias} established that the number of equivalence classes of $d$-dimensional lattice polytopes with bounded volume is finite. In 1992, $\mathrm{B\acute{a}r\acute{a}ny}$ and Pach \cite{barany} improved Arnold’s upper bound to 
\begin{equation*}
\log v(2,m)\ll m^{\frac{1}{3}}. 
\end{equation*}
Furthermore, Bárány and Vershik \cite{Vershik} derived 
\begin{equation}\label{improvedbound1}
    \log v(d,m)\ll m^{\frac{d-1}{d+1}},
\end{equation}
and Bárány \cite{Barany1} obtained
\begin{equation}\label{improvedbound2}
    \log v(d,m)\gg m^{\frac{d-1}{d+1}},
\end{equation}
for any fixed $d$ and sufficiently large $m$.

Let $\kappa(d,w)$ denote the number of distinct classes of $d$-dimensional convex lattice polytopes $P$ with $|P|=w$. In 2011, Liu and Zong \cite{liu2011} addressed  the classification problem for convex lattice polytopes with given cardinality and proved
\begin{equation*}
    \begin{gathered}
        w^{\frac{1}{3}}\ll \log \kappa(2,w)\ll  w^{\frac{1}{3}},\\
        \kappa(d,w)=\infty,\ if\ w\geqslant d+1\geqslant 4.
    \end{gathered}
\end{equation*}
They also provided relevant results regarding Arnold’s problem for centrally symmetric lattice polygons in \cite{liu2011}.

Blanco and Santos \cite{Santos1} demonstrated the finiteness result for $3$-dimensional lattice polytopes of width larger than one and with a given number of lattice points, and performed an explicit enumeration for up to $11$ lattice points \cite{Santos, Blanco}.

By now, all the above asymptotic results are about the order of the number of representatives, rather than the exact limits, even for logarithmic values. Consequently, it remains essential to further analyze the structure of the set of representatives under unimodular equivalence. In this paper, we use affine equivalence as a tool to pursue this analysis.

In this paper, we introduce new generalizations of Arnold's question, with the goal of deepening our understanding of the structure of representatives under unimodular equivalence, particularly for lattice polytopes with the same volume, within a convex body, or under a combination of both conditions.

Specifically, in section 2, we provide a necessary and sufficient condition for affine equivalence, serving as a characterization of this equivalence. Verifying whether two lattice polytopes are affinely equivalent is more challenging than verifying unimodular equivalence, as the invariants under affine equivalence are fewer. The characterization proposed in this paper can lead to a more efficient algorithm for determining affine equivalence between lattice polytopes.

In section 3, we establish a new Arnold-type question based on affine equivalence and prove the right order of the logarithm of the number of representatives under the new definition related to affine equivalence, similar to the forms in \eqref{improvedbound1} and \eqref{improvedbound2}.

In Section 4, we explore the relationship between affine equivalence and unimodular equivalence, proposing a method to subdivide the set of representatives under unimodular equivalence by affine equivalence. These discussions offer a new perspective for addressing representatives under unimodular equivalence.

In section 5, we introduce another Arnold-type theorem involving a convex body. Furthermore, constraints of volume and convex body for lattice polytopes.

In section 6, we present a set of complete invariants for unimodular equivalence in $\mathbb{Z}^2$. Investigating complete invariants for unimodular equivalence is valuable for understanding this relation from a geometric standpoint.

\section{Characterization of affine equivalence}

Still, we let $B^d_r$ denote $n$-dimensional ball with radius $r$.
\begin{defn}
    Let $P_1$ and $P_2$ be $d$-dimensional convex lattice polytopes. If there exists an affine transformation $\tau$ such that $\tau(P_1)=P_1A+\mathbf{v}=P_2$, where $A\in\mathbb{R}^{d\times d}$ and $\mathbf{v}\in\mathbb{R}^d$, then we say $P_1$ and $P_2$ are affinely equivalent.
\end{defn}

Now we consider an invariant borrowed from \cite{Santos1} which can be useful for verifying whether two lattice polytopes are unimodular or affinely equivalent.

\begin{defn}[cf. \cite{Santos1}]\label{defvolvec}
Let $A = \{p_1 \cdots p_n\}$, where $n \geqslant d+1$, be a finite set of points in $\mathbb{Z}^d$. The volume vector of $A$ is defined as
\begin{equation*}
    w = (w_{i_1,...,i_{d+1}}), \quad i_1,...,i_{d+1} \in \mathbb{Z}^{n\choose d+1},
\end{equation*}
where
\begin{equation*}
    w_{i_1,...,i_{d+1}} = \det 
    \begin{pmatrix}
        1 & \cdots & 1\\
        p_{i_1} & \cdots & p_{i_{d+1}}
    \end{pmatrix}.
\end{equation*}
Suppose $w = kw'$, where $k$ is an integer and $w'$ is a primitive vector. Then $w'$ is called the primitive volume vector of $A$.
\end{defn}

The following proposition from M. Blanco and F. Santos demonstrates a strong connection between the volume vector and unimodular equivalence. 

\begin{prop}[cf. \cite{Santos1}]\label{santosvolvec}
    Let $A = \{p_1 \cdots p_n\}$ and $B = \{q_1 \cdots q_n\}$ be two $d$-dimensional finite subsets in $\mathbb{Z}^d$. Suppose that $A$ and $B$ have the same volume vector under some specific order of points in $A$ and $B$. Then there exists a unimodular affine transformation between $A$ and $B$, and when the components of the volume vectors of $A$ and $B$ are coprime, there exists a unimodular transformation between $A$ and $B$.
\end{prop}

Inspired by the proposition above, we conjectured and proved the following theorem, concerning affine equivalence.

\begin{thm}\label{thmaffineiff}
    Two $d$-dimensional convex lattice polytopes $P$ and $Q$ are affinely equivalent if and only if $P$ and $Q$ have the same number of vertices, and after properly ordering the vertices of $P$ and $Q$, the primitive volume vectors of the vertex sets of $P$ and $Q$ are equal.
\end{thm}

\begin{proof}
    Necessity: Suppose that the $d$-dimensional polyhedra $P$ and $Q$ are affinely equivalent. Then clearly, $P$ and $Q$ have the same number of vertices, and the affine transformation between $P$ and $Q$ provides a one-to-one mapping between their vertex sets. Specifically, there exists an affine transformation $\varphi$ such that $\varphi(p_i) = q_i$, for $i = 1, \cdots, n$, where $n = |\mathrm{vert}(P)| = |\mathrm{vert}(Q)|$, and $\mathrm{vert}(P) = \{p_1, \cdots, p_n\}$, $\mathrm{vert}(Q) = \{q_1, \cdots, q_n\}$. Let the volume vectors of $P$ and $Q$ with this vertex ordering be
    \begin{equation*}
        w^P = (w^P_{i_1,...,i_{d+1}}),\ w^Q = (w^Q_{i_1,...,i_{d+1}})\quad i_1,...,i_{d+1}\in \mathbb{Z}^{n\choose d+1}.
    \end{equation*}
    Without loss of generality, assume that $\{p_1, \cdots, p_{d+1}\}$ and $\{q_1, \cdots, q_{d+1}\}$ are affinely independent, otherwise, adjust the positions of $p_i$ and $q_i$ accordingly. Define
    \begin{equation*}
        \mathrm{det}_P \triangleq \det 
    \begin{pmatrix}
        1 & \cdots & 1\\
        p_1 & \cdots & p_{d+1}
    \end{pmatrix},\ \mathrm{det}_Q \triangleq \det 
    \begin{pmatrix}
        1 & \cdots & 1\\
        q_1 & \cdots & q_{d+1}
    \end{pmatrix}.
    \end{equation*}
    Suppose
    \begin{gather*}
        p_i = a_1^i p_1 + \cdots + a_{d+1}^i p_{d+1},\ a_1^i + \cdots + a_{d+1}^i = 1;\\
        q_i = b_1^i q_1 + \cdots + b_{d+1}^i q_{d+1},\ b_1^i + \cdots + b_{d+1}^i = 1;\\
        i = 1, \cdots, n.
    \end{gather*}
    Since $\varphi(p_i) = q_i$, we have
    \begin{equation*}
        a_j^i = b_j^i, \quad i = 1, \cdots, n, \ j = 1, \cdots, d+1.
    \end{equation*}
    Thus:
    \begin{align*}
        w^P_{i_1,...,i_{d+1}} &= \det 
    \begin{pmatrix}
        1 & \cdots & 1\\
        p_{i_1} & \cdots & p_{i_{d+1}}
    \end{pmatrix}\\
    &= \sum_{(k_1,\cdots,k_{d+1}) \in \mathrm{perm}(1,\cdots,d+1)} a^{i_1}_{k_1} a^{i_2}_{k_2} \cdots a^{i_{d+1}}_{k_{d+1}} \cdot \mathrm{det}_P\\
    &= \mathrm{det}_P \cdot
    \sum_{(k_1,\cdots,k_{d+1}) \in \mathrm{perm}(1,\cdots,d+1)} a^{i_1}_{k_1} a^{i_2}_{k_2} \cdots a^{i_{d+1}}_{k_{d+1}}\\
    &\triangleq \mathrm{det}_P
    \cdot \alpha_{i_1,...,i_{d+1}}.
    \end{align*}
    where $\mathrm{perm}(1,\cdots,d+1)$ denotes the set of all permutations of $\{1,\cdots,d+1\}$. Similarly:
    \begin{align*}
        w^Q_{i_1,...,i_{d+1}} &= \mathrm{det}_Q \cdot
    \sum_{(k_1,\cdots,k_{d+1}) \in \mathrm{perm}(1,\cdots,d+1)} b^{i_1}_{k_1} b^{i_2}_{k_2} \cdots b^{i_{d+1}}_{k_{d+1}}\\
    &= \mathrm{det}_Q \cdot
    \sum_{(k_1,\cdots,k_{d+1}) \in \mathrm{perm}(1,\cdots,d+1)} a^{i_1}_{k_1} a^{i_2}_{k_2} \cdots a^{i_{d+1}}_{k_{d+1}}\\
    &= \mathrm{det}_Q
    \cdot \alpha_{i_1,...,i_{d+1}}.
    \end{align*}
    Thus,
    \begin{align*}
        w^P = (w^P_{i_1,...,i_{d+1}}) = \mathrm{det}_P \cdot (\alpha_{i_1,...,i_{d+1}}),\\
        w^Q = (w^Q_{i_1,...,i_{d+1}}) = \mathrm{det}_Q \cdot (\alpha_{i_1,...,i_{d+1}}).
    \end{align*}
    Therefore, $\mathrm{det}_P \cdot w^Q = \mathrm{det}_Q \cdot w^P$, and since both $\mathrm{det}_P$ and $\mathrm{det}_Q$ are integers, this proves the result.

    Sufficiency: Suppose $|\mathrm{vert}(P)| = n = |\mathrm{vert}(Q)|$, and $\mathrm{vert}(P) = \{p_1, \cdots, p_n\}$, $\mathrm{vert}(Q) = \{q_1, \cdots, q_n\}$. Let the volume vectors of $P$ and $Q$ with this vertex ordering be
    \begin{equation*}
        w^P=(w^P_{i_1,...,i_{d+1}}),\ w^Q=(w^Q_{i_1,...,i_{d+1}})\quad i_1,...,i_{d+1}\in \mathbb{Z}^{n\choose d+1},
    \end{equation*}
    and there exist $x$ and $y$ such that $xw^P=yw^Q$. Without loss of generality, assume that \(\{p_1, \cdots, p_{d+1}\}\) and \(\{q_1, \cdots, q_{d+1}\}\) are both affinely independent; otherwise, adjust the positions of \(p_i\) and \(q_i\) simultaneously. Then, the mapping \(p_i \mapsto q_i\) for \(i = 1, \cdots, d+1\) will uniquely determine an affine transformation, denoted by \(\varphi\). Suppose
    \begin{equation}\label{fangshebiaoshi}
        \begin{gathered}
        p_{i} = a_1^{i}p_1 + \cdots + a_{d+1}^{i}p_{d+1}, \quad a_1^i + \cdots + a_{d+1}^i = 1;\\
        q_{i} = b_1^{i}q_1 + \cdots + b_{d+1}^{i}q_{d+1}, \quad b_1^i + \cdots + b_{d+1}^i = 1;\\
        i = 1, \cdots, n.
        \end{gathered}
    \end{equation}
    
    What remains to be shown is
    \begin{equation*}
        a_j^k = b_j^k, \quad k = 1, \cdots, n, \quad j = 1, \cdots, d+1.
    \end{equation*}
    For \(k = 1, \cdots, d+1\), the above equation clearly holds. Now, let \(k > d+1\). Consider
    \begin{align*}
        w^P_{1, \cdots, k, \cdots, d+1} &= \det 
    \begin{pmatrix}
        1 & \cdots & 1 & \cdots & 1\\
        p_1 & \cdots & p_k & \cdots & p_{d+1}
    \end{pmatrix},\\
        w^Q_{1, \cdots, k, \cdots, d+1} &= \det 
    \begin{pmatrix}
        1 & \cdots & 1 & \cdots & 1\\
        q_1 & \cdots & q_k & \cdots & q_{d+1}
    \end{pmatrix},
    \end{align*}
    where \(k\) occupies the \(j\)-th position in \(\{1, \cdots, d+1\}\), meaning that \(k\) replaces \(j\), where \(1 \leq j \leq d+1\). According to equation (\ref{fangshebiaoshi}), we have
    \begin{equation*}
        w^P_{1, \cdots, k, \cdots, d+1} = a^k_j \cdot \det\begin{pmatrix}
        1 & \cdots & 1\\
        p_1 & \cdots & p_{d+1}
    \end{pmatrix}, \quad w^Q_{1, \cdots, k, \cdots, d+1} = b^k_j \cdot \det\begin{pmatrix}
        1 & \cdots & 1\\
        q_1 & \cdots & q_{d+1}
    \end{pmatrix}.
    \end{equation*}
    Since \(xw^P = yw^Q\), it follows that
    \begin{equation*}
        x \cdot \det\begin{pmatrix}
        1 & \cdots & 1\\
        p_1 & \cdots & p_{d+1}
    \end{pmatrix} = y \cdot \det\begin{pmatrix}
        1 & \cdots & 1\\
        q_1 & \cdots & q_{d+1}
    \end{pmatrix}, \quad x \cdot w^P_{1, \cdots, k, \cdots, d+1} = y \cdot w^Q_{1, \cdots, k, \cdots, d+1},
    \end{equation*}
    which immediately implies \(a^k_j = b^k_j\) for \(k = 1, \cdots, n\) and \(j = 1, \cdots, d+1\). Therefore, by the property of affine transformations preserving affine combinations, we have \(\varphi(p_i) = q_i\) for \(i = 1, \cdots, n\).
\end{proof}

\begin{rem}
    \cite{Santos1} tells us that if a lattice polytope $P$ has a volume vector $w^P$ which is primitive, then the affine sublattice generated by $P$, say $L(P)$, will be $\mathbb{Z}^d$. But we are curious about the question in another direction: If $L(P)=\mathbb{Z}^d$, can we say that $w^P$ is primitive? The answer to this question will lead to further understanding of possible affine relationships between different vertices of a lattice polytope: Given any integer vector, how can we determine if it can be the volume vector of a lattice polytope?
\end{rem}

Theorem \ref{thmaffineiff} will lead to an algorithmic improvement. Testing the affine equivalence of two lattice polytopes is more challenging than testing unimodular equivalence, as several invariants under unimodular transformations, such as volume and the number of lattice points, may change under affine transformations. To the authors' knowledge, the only known method to test affine equivalence involves exhaustively considering every pair of simplices separately within the two polytopes and then all possible affine transformations.

However, based on Theorem \ref{thmaffineiff}, to verify the affine equivalence of two lattice polytopes $P$ and $Q$, we can first calculate the primitive volume vectors $w^P$ and $w^Q$ according to an arbitrary order of vertices, keeping track of which vertices correspond to each component of $w^P$ and $w^Q$.

Secondly, arrange the components of $w^P$ and $w^Q$ in ascending or descending order. If the ordered components are not identical, then $P$ and $Q$ cannot be affinely equivalent. If they are identical and assuming $P$ and $Q$ are affinely equivalent, we know that the simplex corresponding to the first component in $w^P$ can only be transformed into one of the simplices corresponding to the components equal to the first in $w^Q$ (or $w^P$). This approach significantly reduces the number of potential correspondence relationships that need to be tested.

\section{Arnold-type question associated with affine equivalence}
In the introduction, we mentioned Arnold's question, which asks the number of representatives under unimodular equivalence with the same volume. It is natural to consider classifying convex lattice polytopes according to other types of equivalence, such as affine equivalence. However, affine transformations generally alter the volume of lattice polytopes, so simply substituting unimodular equivalence with affine equivalence in Arnold's question does not yield a meaningful new question. To refine the problem, we introduce the following lemma.

\begin{lem}\label{Vminiff}
    Let $P$ be a convex lattice polytope in $d$-dimensional space. Denote the set of convex lattice polytopes affinely equivalent to $P$ by $\mathcal{A}(P)$. Then:
    \begin{enumerate}
        \item For any affine equivalence class $\mathcal{A}(P)$, the following value exists and can be attained by some $Q \in \mathcal{A}(P)$:
        \begin{equation*}
            V_{min}(\mathcal{A}(P))\coloneqq \min\{\mathrm{vol}(P):\ P\in\mathcal{A}(P)\}.
        \end{equation*}
        \item For any affine equivalence class $\mathcal{A}(P)$, the polytope $Q \in \mathcal{A}(P)$ achieves $V_{min}(\mathcal{A}(P))$  if and only if the affine sublattice generated by $Q$ is $\mathbb{Z}^d$.
    \end{enumerate}
\end{lem}

\begin{proof}
    To prove the first statement, observe that the set of volumes of convex lattice polytopes in $\mathcal{A}(P)$ forms a subset of all multiples of $1/d!$, which is a discrete, closed set in $\mathbb{R}$. Considering the set
    \begin{equation*}
        \{\mathrm{vol}(P) : P \in \mathcal{A}(P)\} \cap \left[\frac{1}{d!}, \mathrm{vol}(P)\right],
    \end{equation*}
    we obtain a compact set, and the conclusion follows immediately.

    Now, for the sufficiency part of the second assertion, denote the affine sublattice generated by a given lattice polytope $Q$ as $L(Q)$. Suppose the volume of $Q \in \mathcal{A}(P)$ is strictly greater than $V_{min}(\mathcal{A}(P))$, and that $Q'$ attains $V_{min}(\mathcal{A}(P))$. Then, there exists an affine transformation $\phi$ that transforms $Q$ to $Q'$, and consequently, $L(Q)$ to $L(Q')$. Since both $Q$ and $Q'$ are lattice polytopes, $L(Q)$ and $L(Q')$ are affine sublattices of $\mathbb{Z}^d$. Let $M_{\phi}$ denote the matrix corresponding to $\phi$. Given that $\mathrm{vol}(Q) > \mathrm{vol}(Q')$, we have $\det(M_{\phi}) < 1$, implying $L(Q') < L(Q)$, which contradicts the assumption that $L(Q) = \mathbb{Z}^d$ and $L(Q') \subseteq \mathbb{Z}^d = L(Q)$.

    Conversely, suppose that for a convex polytope $Q$, $L(Q)$ (as defined above) is not equal to $\mathbb{Z}^d$. Then there exists an affine transformation $\varphi$ such that $\varphi(L(Q)) = \mathbb{Z}^d$. Since $\det(L(Q)) > \det(\mathbb{Z}^d)$, the matrix corresponding to $\varphi$, $M_{\varphi}$, will have a determinant strictly smaller than 1. Noting that $\varphi(L(Q)) = \mathbb{Z}^d$ and $\mathrm{vert}(Q) \subseteq L(Q)$, where $\mathrm{vert}(Q)$ is the set of vertices of $Q$, we have $\varphi(\mathrm{vert}(Q)) \subseteq \mathbb{Z}^d$. Thus, $\varphi(Q)$ is also a convex lattice polytope affinely equivalent to $Q$, and 
    \begin{equation*}
        \mathrm{vol}(\varphi(Q)) = \det(M_{\varphi}) \cdot \mathrm{vol}(Q) < \mathrm{vol}(Q),
    \end{equation*}
    which contradicts the assumption that $Q$ attains $V_{min}(\mathcal{A}(P))$. Therefore, the necessity part is proved.
\end{proof}

With this lemma in place, we can refine the question mentioned at the beginning of this section as follows:

\begin{center}
    \textit{Given sufficiently large $V$, how many affine equivalence classes $\mathcal{A}(P)$ are there such that $V_{min}(\mathcal{A}(P))=V$?}
\end{center}

We define
\begin{equation*}
    a(d,m)\coloneqq |\{\mathcal{A}(P):\ V_{min}(\mathcal{A}(P))=m/d!\}|,
\end{equation*}
and establish that
\begin{equation*}
    \log a(d,m)\leqslant \log v(d,m)\ll m^{\frac{d-1}{d+1}}.
\end{equation*}
This follows because if we select one representative with volume $m/d!$ from each class in $\{\mathcal{A}(P): V_{min}(\mathcal{A}(P)) = m/d!\}$, forming a set of representatives $rep(d,m)$, then no two lattice polytopes $Q_1, Q_2 \in rep(d,m)$ are affinely equivalent, or particularly unimodular equivalent.

The main task of this part is to prove the following theorem:
\begin{thm}\label{lowerboundadm}
    \begin{equation*}
        m^{\frac{d-1}{d+1}}\ll \log a(d,m).
    \end{equation*}
\end{thm}

Let 
\begin{equation*}
    a'(d,m)\coloneqq|\{\mathcal{A}(P):\ V_{min}(\mathcal{A}(P))\leqslant m/d!\}|.
\end{equation*}
By Theorem \ref{lowerboundadm}, we have the following corollary.
\begin{cor}
    \begin{equation*}
        m^{\frac{d-1}{d+1}}\ll \log a'(d,m)\leqslant \log v(d,m)\ll m^{\frac{d-1}{d+1}}.
    \end{equation*}
\end{cor}

The proof strategy of Theorem \ref{lowerboundadm} is as follows: First, we follow the construction attributed to I. Bárány \cite{Barany1} to generate a sufficiently large number of distinct convex lattice polytopes with volume $m/d!$. Then, we slightly modify each such polytope to ensure they achieve the $V_{min}$ values for their respective classes. Finally, we show that the volume change caused by this modification can be corrected using Bárány's structure. To make this section self-contained, we will also outline Bárány's construction here.

The starting point of Bárány's construction is the convex hull of all integer points in a ball. Let $\mathbb{R}^d_{+}\coloneqq \{(x_1,\dots,x_d) \in \mathbb{R}^d : x_i \geqslant 0 \text{ for all } i\}$. For a convex lattice polytope $Q$ and a vertex $x$ of $Q$, define
\begin{equation*}
    \bigtriangleup_Q(x)=Q\backslash \mathrm{conv}((Q\cap\mathbb{Z}^d)\backslash\{x\}).
\end{equation*}
Let $\mathrm{vert}(Q)$ denote the set of vertices of $Q$. We summarise the main conclusion in \cite{Barany1} as the following proposition.
\begin{prop}[cf. \cite{Barany1}]\label{baranysconstruction}
    For a sufficiently large $V$, if a convex lattice polytope $Q$ in $d$-dimensional space satisfying:
    \begin{enumerate}
        \item $f_0(Q)\gg V^{\frac{d-1}{d+1}}$,
        \item $0\leqslant\mathrm{vol}(Q)-V-\frac{1}{5d!}f_0(Q)\ll V^{\frac{d-1}{2d}}$,
    \end{enumerate}
    then there is a vertices set $X$ of $Q$ containing more than 52 percent of vertices of $Q$, such that for any subset of $X$, denoted by $X'$, not containing the origin and satisfying $|X'|=\lfloor\frac{1}{2}f_0(Q)\rfloor$, there exists a constant $c$ only depending on $d$, such that for $Q$, there are more than $\exp{\{cV^{(d-1)/(d+1)}\}}$ subsets $W$ of $X'$ satisfying
    \begin{equation*}
        \mathrm{vol}\left(Q\backslash\bigcup_{x\in W}\bigtriangleup_{Q}(x)\right)\in\left[V,V+\frac{b}{d!}\right],
    \end{equation*}
    where $b$ is a constant only depending on $d$. We denote the set of all such $Q\backslash\bigcup_{x\in W}\bigtriangleup_{Q}(x)$ by $N_Q$.
\end{prop}
In \cite{Barany1}, Bárány proved that for every sufficiently large $V$, one can always choose a suitable $r$, such that  $Q_r\coloneqq2\mathrm{conv}(B_r^d\cap\mathbb{R}^d_+\cap\mathbb{Z}^d)$ can satisfy the two conditions in Proposition \ref{baranysconstruction}, then he get a set of enough distinct convex lattice polytopes with a volume of nearly $V$. For any polytope $P\in N_{Q_r}$, there is a $t$, $0\leqslant t\leqslant b$, such that $\mathrm{vol}(P)=V+t/d!$. So one can take $P'=P\backslash \mathrm{conv}\{\mathbf{0},\mathbf{e}_1,\cdots,t\mathbf{e}_d\}$ instead of $P$ to obtain convex lattice polytopes with exact volume $V$.

We also need the following lemma in \cite{Barany1}.
\begin{lem}[cf. \cite{Barany1}]\label{Qrvertlem}
    $B_{r-\sqrt{d}}^d\subseteq \mathrm{conv}(B_r^d\cap\mathbb{Z}^d)$.
\end{lem}

Now we will describe the slight adding in every polytope in $N_{Q_r}$. Before that, we need the following lemma.

\begin{lem}\label{Zdiff}
    For a convex lattice polytope $P$, if there are $2d$ vertices of $P$, called $p_1,\cdots,p_{2d}$, which can be repeated, such that $p_1-p_2,\cdots,p_{2d-1}-p_{2d}$ form a basis of $\mathbb{Z}^d$, then the affine sublattice generated by $P$ is $\mathbb{Z}^d$.
\end{lem}

\begin{proof}
    Let the set of all vertices of $P$ be $\mathrm{vert}(P)=\{p_1,\cdots,p_n\}$. Choose a vertex $x$ of $P$ and fix it. For any point $y\in\mathbb{Z}^d$, $y-x$ is also in $\mathbb{Z}^d$. Then there are integers $\alpha_1,\cdots,\alpha_d$ and $\beta_1,\cdots,\beta_n$, such that
    \begin{equation*}
        y-x=\alpha_1(p_1-p_2)+\cdots+\alpha_d(p_{2d-1}-p_{2d})=\beta_1p_1+\cdots+\beta_np_n.
    \end{equation*}
    Obviously $\sum_{i=1}^n\beta_i=0$. Then $y=x+\beta_1p_1+\cdots+\beta_np_n$, which shows that $y$ can be affinely expressed by $p_1,\cdots,p_n$, and the conclusion is yielded.
\end{proof}

Let $Q'_r\coloneqq\mathrm{conv}(B_r^d\cap\mathbb{R}^d_+\cap\mathbb{Z}^d)$, so $Q'_r=(1/2)Q_r$. Let
\begin{equation*}
    p(r)=\max\{x_1:\ (x_1,\cdots,x_d)\in Q_r'\},
\end{equation*}
and $H(x_1=p(r))$ be the affine hyperplane defined by $x_1=p(r)$. Denote by $I_{p(r)}$ the set of integer points in $Q_r'\cap H(x_1=p(r))$, then it is evident that $I_{p(r)}$ consists of all integer points in the positive part of a $(d-1)$-dimensional ball. So one of the following two cases must happen.

\begin{case}
    $I_{p(r)}=\{(p(r),0,\cdots,0)\}$. In this case, we define
    \begin{equation*}
        S_r'\coloneqq \mathrm{conv}\left(Q_r'\cup \left(\left\{\frac{1}{2}\mathbf{0},\frac{1}{2}\mathbf{e}_2,\cdots,\frac{1}{2}\mathbf{e}_d\right\}+(p(r),0,\cdots,0)\right)\right).
    \end{equation*}
    where the sum means adding $(p(r),0,\cdots,0)$ to every vector in the set $\{\mathbf{0},\cdots,\mathbf{e_d}\}$.
\end{case}
\begin{case}
    $(\{\mathbf{0},\mathbf{e}_2,\cdots,\mathbf{e}_d\}+(p(r),0,\cdots,0))\subseteq I_{p(r)}$. In this case, we define
    \begin{equation*}
        S_r'\coloneqq \mathrm{conv}\left(Q_r'\cup \left(\left\{\frac{1}{2}\mathbf{0},\frac{1}{2}\mathbf{e}_2,\cdots,\frac{1}{2}\mathbf{e}_d\right\}+(p(r)+\frac{1}{2},0,\cdots,0)\right)\right).
    \end{equation*}
\end{case}

It is evident that in the above two cases, $S_r'$ is no longer a lattice polytope. We prove the following lemma which is based on Lemma \ref{Qrvertlem} and similar to Lemma 2.3 in \cite{Barany1}.
\begin{lem}\label{SrandQr}
    In the above both cases, $\mathrm{vol}(S_r')-\mathrm{vol}(Q_r')\ll r^{(d-1)/2}$.
\end{lem}

\begin{proof}
    Consider $Q_{r+2}'$, then the cap $\{x_1\geqslant p(r+1)\}\cap Q_{r+1}'$ (or $\{x_1\geqslant p(r)\}\cap Q_{r+1}'$) has width larger than or equal to 1. Then $Q_{r+2}'\cap H(x_1=p(r+1))$ (respectively $Q_{r+2}'\cap H(x_1=p(r))$) contains $\{\mathbf{0},\mathbf{e}_2,\cdots,\mathbf{e}_d\}+(p(r),0,\cdots,0)$ (respectively $\{\mathbf{0},\mathbf{e}_2,\cdots,\mathbf{e}_d\}+(p(r+1),0,\cdots,0)$). So we know that in both cases, $Q_{r+2}'$ contains $S_r'$.

    In case 1, if we choose a hyperplane $H$ which strictly separates $\frac{1}{2}\mathbf{e}_i+(p(r),0,\cdots,0)$ and $Q_r'$ and let the halfspace not containing $Q_r'$ be $H^+$, then by Lemma \ref{Qrvertlem}, the width of the cap $C\coloneqq H^+\cap Q_{r+2}'$ is not larger than $\sqrt{d}+2$. The diameter of $C$ is at most
    \begin{equation*}
        2\sqrt{(r+2)^2-[r+2-(\sqrt{d}+2)]^2}=2\sqrt{(4+2\sqrt{d})r+4-d}\leqslant 4\sqrt{dr}
    \end{equation*}
    for $d,r\geqslant2$. So every such $C$, generated by $H^+$ related to $\frac{1}{2}\mathbf{e}_i+(p(r),0,\cdots,0)$, will be contained in a cap $C'$ whose center is located on the line connecting the origin and $\frac{1}{2}\mathbf{e}_i+(p(r),0,\cdots,0)$, with the width $4\sqrt{dr}$. The volume of this cap is $\ll r^{(d-1)/2}$, and we have $\bigtriangleup_{S_r'}\left(\frac{1}{2}\mathbf{e}_i+(p(r),0,\cdots,0)\right)\subseteq C'$. The same discussion can be conducted on other $\frac{1}{2}\mathbf{e}_i+(p(r),0,\cdots,0)$, then we know that 
    \begin{equation*}
        \mathrm{vol}\bigcup_{i=1}^d\bigtriangleup_{S_r'}\left(\frac{1}{2}\mathbf{e}_i+(p(r),0,\cdots,0)\right)\ll r^{(d-1)/2}.
    \end{equation*}

    For case 2, since we have already shown that $S_r'\subseteq Q_{r+2}'$, the above discussion for case 1 is also true for case 2. So the whole lemma is proved.
\end{proof}

Let $S_r\coloneqq 2S_r'$, then $S_r\supseteq Q_r$. It is easy to see that
\begin{equation*}
    (S_r\cap\mathbb{Z}^d)\backslash (Q_r\cap\mathbb{Z}^d)=
    \begin{cases}
        \left\{\mathbf{e}_2,\cdots,\mathbf{e}_d\right\}+(2p(r),0,\cdots,0) & \text{in case 1}\\
        \left\{\mathbf{e}_2,\cdots,\mathbf{e}_d\right\}+(2p(r)+1,0,\cdots,0) & \text{in case 2}
    \end{cases},
\end{equation*}
and the points listed on the right-hand side of the above equation are vertices of $S_r$. We denote
\begin{equation*}
    B\coloneqq\begin{cases}
        \left\{\mathbf{0},\mathbf{e}_2,\cdots,\mathbf{e}_d\right\}+(2p(r),0,\cdots,0) & \text{in case 1}\\
        \left\{\mathbf{0},\mathbf{e}_2,\cdots,\mathbf{e}_d\right\}+(2p(r)+1,0,\cdots,0) & \text{in case 2}
    \end{cases}.
\end{equation*}

Recall that for every sufficiently large $V$, \cite{Barany1} tells us that one can always choose a $Q_r$ satisfying the two conditions in Proposition \ref{baranysconstruction}. First we take $V+1/d!$ as the $V$ in Proposition \ref{baranysconstruction} and choose $Q_r$ satisfying the conditions in \ref{baranysconstruction} for $V+1/d!$, where $V$ is sufficiently large. By the second condition in Proposition \ref{baranysconstruction}, we have $r^d\ll V$, i.e., $r^{(d-1)/2}\ll V^{(d-1)/2d}$. So by Lemma \ref{SrandQr}, we know that $S_r$ according to $Q_r$ satisfies the two conditions in Proposition \ref{baranysconstruction}. Then under the assumption that $V$ and $r$ are sufficiently large, replacing $Q_r$ by $S_r$, we can always choose an $X'$ in $X$ (as defined in Proposition \ref{baranysconstruction}), not containing the origin and any vertex in $B$, and satisfying $|X'|=\lfloor\frac{1}{2}f_0(S_r)\rfloor$. By shaving the subsets of vertices in $X'$, as shown in Proposition \ref{baranysconstruction}, we can get at least $\exp\{cV^{(d-1)/(d+1)}\}$ convex lattice polytopes with the volume between $V+1/d!$ and $V+b/d!$, where $c$ and $b$ depend only on $d$. As the previous notation in Proposition \ref{baranysconstruction}, we denote the set of those polytopes by $N_{S_r}$.

Notice that every polytope $P$ in $N_{S_r}$ preserves the vertices set $B$, so in both two cases of $S_r$, we know that $\left\{\mathbf{e}_2,\cdots,\mathbf{e}_d\right\}$ can be written as differences of vertices of $P$. As the description after Proposition \ref{baranysconstruction}, for every polytope $P$ in $N_{S_r}$, one can take $P'=P\backslash\mathrm{conv}\{\mathbf{0},\mathbf{e}_1,\cdots,t\mathbf{e}_d\}$ for some suitable $t$, $1\leqslant t\leqslant b$, such that $\mathrm{vol}(P')=V$. Notice that $\mathbf{e}_1$ and $\mathbf{e}_2$ still are vertices of $P'$, so $\mathbf{e}_1-\mathbf{e}_2$ along with $\left\{\mathbf{e}_2,\cdots,\mathbf{e}_d\right\}$ can form a basis of $\mathbb{Z}^d$, showing that there is a basis of $\mathbb{Z}^d$ such that every basis vector can be expressed as the difference of vertices of $P'$. By Lemma \ref{Zdiff} and \ref{Vminiff}, this leads to a sufficiently large amount of polytopes $P'$ with volume $V$ and attaining $V_{min}$ for their classes.

It is only left to show that nearly all those $P'$'s are affinely non-equivalent to each other. That is easy: all those $P'$'s are in the ball $B^d_{2r+4}$, and the number of possible affine transformations between two lattice polytopes in $B^d_{2r+4}$ is $\ll (2r+4)^{d^2+2d}\ll r^{d^2+2d}$. Since the number of $P'$'s is greater than $\exp\{cV^{(d-1)/(d+1)}\}$ for some $c$ depending only on $d$, for sufficiently large $r$, the number of $P'$'s that are not affinely equivalent to each other will also be larger than $\exp\{c'V^{(d-1)/(d+1)}\}$ for some $c'$ depending only on $d$. Then we finish the proof of Theorem \ref{lowerboundadm}.

\section{Relationship between the number of representatives for different equivalences}
\subsection{The case of representatives in a ball}
Now we discuss the case of representatives in a ball under the two different equivalence relations. We define
\begin{align*}
     K_r^d:&\ \text{the set of representatives of convex lattice polytopes in $B^d_r$ for unimodular equivalence},\\
     A_r^d:&\ \text{the set of representatives of all convex lattice polytopes in $B_r^d$ for affine equivalence},\\
     H_r^d:&\ \text{the set of all convex lattice polytopes in $B_r^d$},\\
     LP_r^d:&\ B_r^d\cap\mathbb{Z}^d,\\
     P_r^d:&\ \text{the convex hull of $B_r^d\cap\mathbb{Z}^d$, i.e. $\mathrm{conv}(B_r^d\cap\mathbb{Z}^d)$},\\
     V_r^d:&\ \text{the set of vertices of $P_r^d$, i.e. $\mathrm{vert}(P_r^d)$}.
\end{align*}

\begin{lem}[cf. \cite{Barany2}]\label{barafkbounds}
    For any $ d \geqslant 2 $, we have:
    \begin{equation*}
        r^{d\frac{d-1}{d+1}}\ll f_k(P_r^d)\ll r^{d\frac{d-1}{d+1}},
    \end{equation*}
    where $ f_k $ refers to the number of $ k $-dimensional facets of the polyhedron. In particular,
    \begin{equation}\label{maxballver}
        r^{d\frac{d-1}{d+1}}\ll f_0(P_r^d)\ll r^{d\frac{d-1}{d+1}},
    \end{equation}
    where $ f_0 $ denotes the number of vertices. The constants involved in these expressions depend only on the dimension $ d $.
\end{lem}

Using the above lemma, we have the following theorem.

\begin{thm}\label{relequiv}
Define
\begin{equation*}
    PV_{d,r}\coloneqq \{\varphi\ is\ an\ affine\ transformation:\ \exists P_1,P_2\in H_r^d\ s.t.\ \varphi(P_1)=P_2\},
\end{equation*}
then
\begin{enumerate}
    \item  $|PV_{d,r}|\ll r^{d^2+2d}$
\item $\lim_{r\rightarrow\infty}\log|H_r^d|:\log|K_r^d|:\log|A_r^d|=1.$
\end{enumerate}
\end{thm}

\begin{proof}
    By the definitions of $H_r^d$, $K_r^d$ and $A_r^d$, it is evident that $|H_r^d|\geqslant|K_r^d|\geqslant|A_r^d|$, so we only need to prove
    \begin{equation*}
        \lim_{r\rightarrow\infty}\frac{\log|H_r^d|}{\log|A_r^d|}=1.
    \end{equation*}

    Define
    \begin{equation}\label{AMdr}
        AM_{d,r}\coloneqq\{A\in\mathbb{R}^{d\times d}:\ \exists P_1,P_2\in H_r^d\ and\ \mathbf{v}\in\mathbb{R}^d\ s.t.\ P_1=P_2A+\mathbf{v}\}.
    \end{equation}
    If $A\in AM_{d,r}$, then there exist two simplices in $H_r^d$, denoted by $S_1, S_2\in\mathbb{R}^{d\times (d+1)}$, where columns of $S_i,\ i=1,2$ are corresponding to the vertices, such that $S_1A+\mathbf{v}=S_2$. So $AM_{d,r}$ can be embedded one-to-one into the following set:
    \begin{equation*}
        SP_{d,r}\coloneqq\{(S_1,S_2):\ S_1,S_2\ are\ simplices\ in\ H_r^d\}.
    \end{equation*}
    The number of lattice points contained in $B_r^d$ is $O(r^d)$, so the number of simplices in $H_r^d$ is at most
    \begin{equation*}
        \binom{O(r^d)}{d+1}
        \leqslant O((r^d)^{d+1})=O(r^{d^2+d}),
    \end{equation*}
    which is also an upper bound of $|AM_{d,r}|$. So an upper bound of $PV_{d,r}$ is $|AM_{d,r}|\cdot O(r^d)\ll r^{d^2+2d}$, where $O(r^d)$ is the order of the number of possible $\mathbf{v}$ in (\ref{AMdr}).

    It is obvious that
    \begin{equation*}
        |H_r^d|\geqslant|A_r^d|\geqslant\frac{|H_r^d|}{|AM_{d,r}|}.
    \end{equation*}
    Take the logarithm on the above inequalities and assume that $r$ is sufficiently large, we get
    \begin{equation*}
        \log|H_r^d|\geqslant\log|A_r^d|\geqslant\log|H_r^d|-\log|AM_{d,r}|\geqslant\log|H_r^d|-(d^2+d)\log r+c(d),
    \end{equation*}
    where $c(d)$ is a constant that only depend on $d$. Divide all terms in the last equation by $\log|H_r^d|$, then we have
    \begin{equation*}
        1\geqslant\frac{\log|A_r^d|}{\log|H_r^d|}\geqslant1-\frac{(d^2+d)\log r+c(d)}{\log|H_r^d|}.
    \end{equation*}
    By Lemma \ref{barafkbounds}, let $\mathrm{vert}(P)$ be the set of vertices of a lattice polytope $P$. We can construct lattice polytopes by considering
    \begin{equation}\label{eqpolyform}
        \mathrm{conv}\left[\left(LP_r^d\right)\backslash W\right],
    \end{equation}
    where $W\subseteq V_r^d$. Since there are $2^{f_0(P_r^d)}$ different subsets $W$ in $V_r^d$, we can get $2^{f_0(P_r^d)}$ different polytopes having form (\ref{eqpolyform}) and lying in $B_r^d$. So
    $\log|H_r^d|\geqslant f_0(P_r^d)\gg r^{d\frac{d-1}{d+1}}$. Take the limitation on the above inequality and the conclusion is yielded.
\end{proof}

\begin{rem}\label{remrelequiv}
    One can define $H_r^d$, $A_r^d$ and $K_r^d$ in the same way as in the above theorem, and replace $B_r^d$ by $rK$ for any other convex body $K$, then the same theorem will still hold.
\end{rem}
\begin{rem}
    Theorem \ref{relequiv} shows that even though we allow the lattice polytopes in $B_r^d$ to be affinely transformed into smaller a ball, the number of resulting polytopes will not be changed in the sense of logarithm.
\end{rem}
\begin{rem}
    Theorem \ref{relequiv} provides a simplification of Bárány's work on volume \cite{Barany1} and Zong's work on cardinality \cite{C. Zong3}. Specifically, once their constructions of polytopes with volume $m$ or cardinality $w$ can be contained within a convex body whose volume is polynomially related to $m$ or $w$, then the analysis of unimodular non-equivalence can be omitted from the discussion of the orders of $\log v(d,m)$ and $\log \kappa(d,w)$.
\end{rem}

Furthermore, we introduce an additional result of the bounds for $A\in AM_{d,r}$, see (\ref{AMdr}).
\begin{thm}
    For any $A\in AM_{d,r}$, let $a_i$ be the $i$-th row of $A$, then
    \begin{equation*}
        \|a_i\|\ll r^d.
    \end{equation*}
\end{thm}

\begin{proof}
    As mentioned before, there exist two simplices in $H_r^d$, denoted by $S_1, S_2\in\mathbb{R}^{d\times (d+1)}$, where columns of $S_i,\ i=1,2$ are corresponding to the vertices, such that $S_1A+\mathbf{v}=S_2$. By applying a suitable translation, we can assume that there are $S_1,S_2\in\mathbb{R}^{d\times d}$, with the property that the Euclidean length of any column of both $S_1$ and $S_2$ is less than $2r$, and that $S_1A=S_2$. So we have $A=S_1^{-1}S_2$, i.e.,
    \begin{equation*}
        A=\frac{1}{\det(S_1)}
        \begin{bmatrix}
            D_{11} & D_{21} & \cdots & D_{d1}\\
            D_{12} & D_{22} & \cdots & D_{d2}\\
            \vdots & \vdots & \ddots & \vdots \\
            D_{1d} & D_{2d} & \cdots\ & D_{dd}
        \end{bmatrix}
        \begin{bmatrix}
            s_1^{(2)}\\
            s_2^{(2)}\\
            \vdots\\
            s_d^{(2)}
        \end{bmatrix},
    \end{equation*}
    where $D_{ji}$ is the $(i,j)$-cofactor of $S_1$, $s_k^{(2)}$ is the $k$-th row of $S_2$.

    Since every row vector of $S_1$ is within $O_{d,2r}$, each row vector of any $D_{ji}$, for $i,j=1,\cdots,d$, will lie in $O_{d-1,2r}$. This follows because each minor $M_{ji}$ of $S_1$ can be seen as a simplex (the columns are according to vertices) in $(d-1)$ dimensions, formed by projecting all columns of $S_1$ except the $i$-th column onto the hyperplane $\{x_j=0\}$. So we have
    \begin{equation*}
        |D_{ji}|\leq d!v(O_{d-1,2r})\coloneqq c_1(d)r^{d-1},\ 1\leq i,j\leq d
    \end{equation*}
    where $c_1(d)$ depends only on $d$. It follows that
    \begin{equation*}
        \|a_i\|=\frac{1}{\det(S_1)}\|D_{1i}s_1^{(2)}+\cdots+D_{di}s_d^{(2)}\|\leq\frac{1}{\det(S_1)} 2dc_1(d)r^d\leq c_2(d)r^d,
    \end{equation*}
    where $c_2(d)\coloneqq 2dc_1(d)$ depends only on $d$ as well.
\end{proof}

\subsection{The case of representatives with the same volume}
First, we discuss the case of representatives with the same volume under the two equivalence relations. We denote by $A_V^d$ the set of representatives of all convex lattice polytopes with $V_{min}$ value $V$ for affine equivalence. Without loss of generality, we assume that $V$ is the normalized volume, i.e., the real volume of polytopes in $A_V^d$ is $V/d!$.

We define a set of convex lattice polytopes with volume $V$ and unimodular non-equivalent to each other as follows:
\begin{equation*}
    L_V^d\coloneqq\bigcup_{i=1}^V (i\cdot A_{V/i}^d),
\end{equation*}
where the product on the right-hand side means conducting the product on each polytope in $A_{V/i}^d$. All polytopes in this $L_V^d$ have volume $V$. Furthermore, let $K^d_V$ be a set of representatives with normalized volume $V$ under unimodular equivalence. We have the following theorem.

\begin{thm}\label{relatKandV}
    \begin{equation*}
        \lim_{m\rightarrow\infty}\frac{\log|K_V^d|}{\log|L_V^d|}=1.
    \end{equation*}
\end{thm}

To prove this theorem, we need the following lemma due to I. Bárány and A.M. Vershik \cite{Vershik}.

\begin{lem}[cf. \cite{Vershik}]
    There exists a constant $c$ making the following statement true. For any convex lattice polytope $P$ with large enough volume, there is another convex lattice polytope $Q$ and a vector $\gamma=(\gamma_i)\in\mathbb{R}^d$, such that $Q$ is unimodular equivalent to $P$ and lies in
    \begin{equation*}
        T(\gamma)\coloneqq\{x=\sum_{i=1}^d\xi\mathbf{e}_i\in\mathbb{R}^d:\ 0\leqslant\xi_i\leqslant \gamma_i\},
    \end{equation*}
    where $\mathbf{e}_i$ is the $i$-th vector of the orthonormal basis of $\mathbb{R}^d$, and $\gamma$ satisfies $\prod_{i=1}^d\gamma_i\leqslant c\cdot \mathrm{vol}(P)$.
\end{lem}

By the above lemma, we have:
\begin{lem}\label{lembox}
    There is a constant $c$, such that for any convex lattice polytope $P$ with volume $V$, there exists a lattice polytope $Q$ that is unimodular equivalent to $P$ and lies in the cube 
    \begin{equation*}
        T(cV)\coloneqq\{x=\sum_{i=1}^d\xi\mathbf{e}_i\in\mathbb{R}^d:\ 0\leqslant\xi_i\leqslant cV\}.
    \end{equation*}
\end{lem}

Now we turn to prove Theorem \ref{relatKandV}.

\begin{proof}[Proof of Theorem \ref{relatKandV}]
Without loss of generality, we let polytopes in $L_V^d$ be in $T(cV)$.

We can easily show that the polytopes in $L_V^d$ are unimodular non-equivalent to each other. By Lemma \ref{Vminiff} we know that there is a unique $V_{min}$ for an affine equivalence class. So polytopes in different $A_{V/i}^d$ for different $i$ will be in different affine equivalence classes, which leads to that polytopes in different $i\cdot A_{V/i}^d$ for different $i$ will not be unimodular equivalent. Furthermore, by definition of $A_{V/i}^d$, the polytopes in the same $A_{V/i}^d$ are affinely non-equivalent, so those in the same $i\cdot A_{V/i}^d$ are unimodular non-equivalent.

Suppose $K^d_V$ containing $L_V^d$ and also lying in $T(cV)$. By Theorem \ref{relequiv}, Remark \ref{remrelequiv} and the fact that $T(cV)=cV\cdot T(1)$, we have
\begin{equation*}
    |PV_{d,r}(T(1))|\ll (cV)^{d^2+2d},
\end{equation*}
where
\begin{equation*}
    PV_{d,r}(T(1))\coloneqq \{\varphi\ \text{is an affine transformation: } \exists\ \text{convex lattice polytope}\  P_1,P_2\subseteq T(1)\ s.t.\ \varphi(P_1)=P_2\}.
\end{equation*}

For any convex lattice polytope $P$ with normalized volume $V$, let the affine sublattice generated by $P$ be $L(P)$ and suppose $[L(P):\mathbb{Z}^d]=i\leqslant V$. Then there will be a polytope in $A^d_{V/i}\subseteq L_V^d$ affinely equivalent to $P$. By $\log|K_V^d|\geqslant\log|L_V^d|\gg V^{(d-1)/(d+1)}$, we have
\begin{equation*}
    \frac{\log|K_V^d|}{\log|L_V^d|}\ll\frac{\log(|L_V^d|\cdot cV^{d^2+2d})}{\log|L_V^d|}=1+\frac{\log(cV^{d^2+2d})}{\log|L_V^d|}\rightarrow1\quad(V\rightarrow\infty).
\end{equation*}
\end{proof}

\begin{rem}
    If we can further show that
    \begin{equation}\label{eqlogLandA}
        \log|L_V^d|/\log|A_{V}^d|\rightarrow1\ ( V\rightarrow\infty),
    \end{equation}
    or
    \begin{equation}\label{eqLandA}
        |L_V^d|/|A_{V}^d|\rightarrow1\ ( V\rightarrow\infty),
    \end{equation}
    then the above theorem will yield that $\log|K_V^d|/\log|L_{V}^d|\rightarrow1\ (V\rightarrow\infty)$, which shows that nearly all (in the sense of logarithm) representatives with volume $V$ under unimodular equivalence can not be further "shrunk" by affine transformations. Since
    \begin{equation*}
        |L_V^d|=\sum_{i=1}^V|A_{V/i}^d|,
    \end{equation*}
    equation (\ref{eqLandA}) is equivalent to
    \begin{equation*}
        \left(\sum_{i=2}^V|A_{V/i}^d|\right)/A_V^d\rightarrow0\ (V\rightarrow\infty),
    \end{equation*}
    which seems to be true, because $(V-V/i)\rightarrow\infty$ when $V\rightarrow\infty$ for any $1\leqslant i\leqslant V$. Or if 
    \begin{equation*}
        \lim_{V\rightarrow\infty}\frac{\log|A_V^d|}{V^{(d-1)/(d+1)}}
    \end{equation*}
    exists, then equation (\ref{eqlogLandA}) will be possible to be proved. So we propose the following problem and conjectures.
\end{rem}

\begin{prob}
    Does
    \begin{equation*}
        \lim_{V\rightarrow\infty}\frac{\log|A_V^d|}{V^{(d-1)/(d+1)}}
    \end{equation*}
    exist or not? If it exists, determine this limit.
\end{prob}

\begin{conj}
    Let $K_V^d$ and $L_V^d$ be defined as above. Then
    \begin{equation*}
        \log|K_V^d|/\log|L_{V}^d|\rightarrow1\ (V\rightarrow\infty).
    \end{equation*}
\end{conj}

\begin{conj}
    Let $A_V^d$ and $L_V^d$ be defined as above. Then
    \begin{equation*}
        |L_V^d|/|A_{V}^d|\rightarrow1\ (V\rightarrow\infty),
    \end{equation*}
    or
    \begin{equation*}
        \left(\sum_{i=2}^V|A_{V/i}^d|\right)/A_V^d\rightarrow0\ (V\rightarrow\infty).
    \end{equation*}
    A weaker version of this conjecture is 
    \begin{equation*}
        \log|L_V^d|/\log|A_{V}^d|\rightarrow1\ (V\rightarrow\infty).
    \end{equation*}
\end{conj}

\begin{rem}
    One might expect that $L_V^d$ itself could be a set of representatives with volume $V$ for unimodular equivalence. But this is not always true. Since $L_V^d$ forms a set of representatives of all lattice polytopes with volume not larger than $V$ for affine equivalence. So for any convex lattice polytope $P_1$ with volume $V$, there will be a unique polytope $P_2$ in $L_V^d$ such that $P_1$ and $P_2$ are affinely equivalent. Then if there is a polytope in $L_V^d$ that is unimodular equivalent to $P_1$, it must be $P_2$. But $P_1$ and $P_2$ can be not unimodular equivalent. For instance, every two simplices with the same volume are affinely equivalent, but they can be not unimodular equivalent. Specifically, the following two lattice polygons in $\mathbb{Z}^2$ have the same volume:
\begin{equation*}
    P_1=\begin{bmatrix}
        9 & 0\\
        0 & 10
    \end{bmatrix}
    \ and\ 
    P_2=\begin{bmatrix}
        6 & 0\\
        0 & 15
    \end{bmatrix},
\end{equation*}
where every column represents a vertex. But we have
\begin{equation*}
    P_1^{-1}P_2=
    \begin{bmatrix}
        \frac{2}{3} & 0\\
        0 & \frac{3}{2}
    \end{bmatrix},
\end{equation*}
which is not a unimodular matrix.

In fact, $L_V^d$ can form a set of representatives of lattice polytopes with volume $V$ for "unimodular affine equivalence", i.e., two lattice polytopes $P_1$ and $P_2$ are unimodular affine equivalent if and only if there is an affine transformation with determinant 1 that transform $P_1$ to $P_2$.
\end{rem}

\section{Classification of convex lattice polytopes in a ball}
Now we introduce another lemma and then yield the bounds of $|K_r^d|$.

\begin{lem}[cf. \cite{Vershik}]
    Let $ v(d,m) $ be the set of equivalence classes of convex lattice polyhedra in $ d $-dimensional space with normalized volume at most $ m $. Then we have:
    \begin{equation}\label{volnum}
        \log v(d, m) \ll m^{\frac{d-1}{d+1}},
    \end{equation}
    where the constant involved depends only on the dimension $ d $.
\end{lem}

\begin{thm}\label{bounds}
    Let $r$ be an integer and $K_{d,r}$ be defined as above, then we have
    \begin{equation*}
        r^{d\frac{d-1}{d+1}}\ll\log|A_r^d|\leqslant\log|K_r^d|\ll r^{d\frac{d-1}{d+1}},
    \end{equation*}
    where the implied coefficients both depend only on $d$.
\end{thm}

\begin{proof}
    First, we prove the lower bound. By Theorem \ref{relequiv}, we have
    \begin{equation*}
        \log|K_r^d|\geqslant\log|A_r^d|\geqslant\log\frac{|H_r^d|}{|PV_{d,r}|}\gg r^{d\frac{d-1}{d+1}}-\log r^{d^2+2d}\gg r^{d\frac{d-1}{d+1}}.
    \end{equation*}
    
    Next we prove the upper bound. The volume of polyhedra in $ O_{d,r} $ is less than the volume of the sphere
    \begin{equation*}
        \mathrm{vol}(B_r^d)=\frac{\pi^{d/2}}{\Gamma(1+\frac{d}{2})}r^d.
    \end{equation*}
    Thus, by equation (\ref{volnum}), we have
    \begin{align*}
        \log\left|K_r^d\right| &\leqslant \log v(d,d!v(B_r^d)) \\
        &\ll d!v(O_{d,r})^{\frac{d-1}{d+1}} \\
        &=c(d)r^{d\frac{d-1}{d+1}} \\
        &\ll r^{d\frac{d-1}{d+1}}.
    \end{align*}
    where
    \begin{equation*}
        c(d)=d!\frac{\pi^{d/2}}{\Gamma(1+\frac{d}{2})}
    \end{equation*}
    is a constant depending only on the dimension $ d $. The upper bound is proved.
\end{proof}

Next, we discuss the lattice polytopes in $B_r^d$ that have specific volumes. Let $H_r^d(V)$ be the set of lattice polytopes which lie in $B_r^d$ and have volume $V/d!$. To make our discussion meaningful, we suppose $V/d!\leqslant\mathrm{vol}(B_r^d)$. We introduce the following lemma.
\begin{lem}[cf. \cite{Barany1}]
    \begin{equation*}
        B^d_{r-2\sqrt{d}}\subseteq\mathrm{conv}(LP_r^d\backslash V_r^d).
    \end{equation*}
\end{lem}

Then we know the following straight conclusion.
\begin{thm}
    There exists a variable $c(r)=O(r^{d-1})=o(r^d)$, such that 
    \begin{equation*}
        V^{\frac{d-1}{d+1}}\ll \log |H_r^d\left(V/d!\geqslant \mathrm{vol}(B_r^d)-c(r)\right)|,
    \end{equation*}
    where $H_r^d(V/d!\geqslant \mathrm{vol}(B_r^d)-c(r))$ means the subset of $H_r^d$ consisting of those polytopes with volume larger than or equal to $\mathrm{vol}(B_r^d)-c(r)$.
\end{thm}

\begin{proof}
    By the above lemma, we know that for large enough $r$ and any subset $W\subseteq V_r^d$, 
    \begin{equation*}
        \mathrm{conv}\left[(LP_r^d)\backslash W\right]\supseteq B_{r-2\sqrt{d}}^d.
    \end{equation*}
    Define
    \begin{equation*}
        c(r)=\mathrm{vol}(B_r^d)-\mathrm{vol}(B_{r-2\sqrt{d}}^d),
    \end{equation*}
    then $c(r)=O(r^{d-1})=o(r^d)$. So all lattice polytopes having the form $\mathrm{conv}\left[(LP_r^d)\backslash W\right]$, $W\subseteq V_r^d$, will have the volume larger than or equal to $\mathrm{vol}(B_r^d)-c(r)$.
\end{proof}

By the same idea and the construction described in Section 3, it is not hard to prove:
\begin{thm}
    There exists a variable $c(r)=o(r^d)$, such that 
    \begin{equation*}
        V^{\frac{d-1}{d+1}}\ll \log |H_r^d\left(V\right)|,
    \end{equation*}
    where $H_r^d(V)$ means the subset of $H_r^d$ consisting of those polytopes with volume exactly equal to $V\leqslant\mathrm{vol}(B_r^d)-c(r)$.
\end{thm}

So we can see from the above two theorems that:
\begin{enumerate}
    \item There are a lot of lattice polytopes in $B_r^d$ with nearly volume $B_r^d$;
    \item The constraint of a convex body $B_r^d$ hardly effect the order of $\log v(d,m)$.
\end{enumerate}
So we propose the following conjecture.
\begin{conj}
    There exists a variable $c(r)=o(r^d)$, such that
    \begin{equation*}
        \lim_{r\rightarrow\infty}\frac{\log |H_r^d\left(V/d!\geqslant \mathrm{vol}(B_r^d)-c(r)\right)|}{\log H_r^d}=1.
    \end{equation*}
\end{conj}

\section{Complete invariants for unimodular equivalence in $\mathbb{Z}^2$}
In \cite{Santos1}, F. Santos called the volume vector as an \textit{almost complete} invariant for unimodular equivalence, due to Proposition \ref{santosvolvec}. Now we will show that the volume vector, along with another useful invariant, constitutes a set of complete invariants for unimodular equivalence in 2-dimensions. Before that, we need the following definition.

\begin{defn}
    Let $A=\{p_1\cdots p_n\}$, where $n \geqslant d+1$, be a finite set of points in $\mathbb{Z}^d$. Suppose the equation of the hyperplane formed by $p_{i_1},...,p_{i_d}$ is $\omega_{i_1,...,i_d}x+c_{i_1,...,i_d}=0$, where $\omega_{i_1,...,i_d}\in\mathbb{R}^d$ is a primitive vector and $c_{i_1,...,i_d}$ is an integer. Then we define
    \begin{equation*}
        D_i=(d_{i_1,...,i_d}),\quad i_1,...,i_d\in \mathbb{Z}^{n-1\choose d},
    \end{equation*}
    where
    \begin{equation*}
        i_1,...,i_d\neq i,\quad d_{i_1,...,i_d}=\omega_{i_1,...,i_d}p_i+c_{i_1,...,i_d}.
    \end{equation*}
    The lattice height vector of $A$ is defined as $D=(D_i),\ i=1,...,n$.
\end{defn}

Now we restrict ourselves to the 2-dimensional case.

\begin{thm}\label{thmcompleteinv}
    Let $A=\{p_1\cdots p_n\}$ and $B=\{q_1\cdots q_n\}$ be two finite 2-dimensional subsets in $\mathbb{Z}^2$. Suppose that under a certain specific order of points in $A$ and $B$, $A$ and $B$ have the same volume vector and lattice height vector. Then there exists a unimodular transformation between $A$ and $B$.
\end{thm}

We use the following lemma to prove the above theorem first.

\begin{lem}\label{lemcompleteinv}
    Let $A=\{\mathbf{0},p_1,p_2\}$ and $B=\{\mathbf{0},q_1,q_2\}$ be two simplices in 2-dimensional space. If $A$ and $B$ have the same volume and lattice height vector (where, when calculating the lattice height, it is always assumed that the normal vector of the face points towards the side of the simplex, standardizing the sign of the expression), then $A$ and $B$ are unimodularly equivalent.
\end{lem}

\begin{proof}[Proof of Theorem \ref{thmcompleteinv}]
    The main idea is to notice that the volume vector can encode the affine (or linear) relationship between different vertices.
    
    If $A$ and $B$ are defined as in Theorem \ref{thmcompleteinv}, then it can be deduced that $A_1\coloneqq \{p_1,\cdots,p_3\}$ and $B_1\coloneqq \{q_1,\cdots,q_3\}$ also have the same volume and lattice height vector. So by considering Lemma \ref{lemcompleteinv} and noticing that the volume and lattice height vector are invariant under lattice translation, we know that $A_1$ is unimodular equivalent to $B_1$, i.e., there exists a unimodular matrix $U$ and an integer point $\mathbf{v}$, such that $A_1U+\mathbf{v}=B_1$. Now it is left to show that $p_iU+\mathbf{v}=q_i$ for other $i=4,\cdots,n$.

    Recall that in the proof of Theorem \ref{thmaffineiff}, if we assume that
    \begin{gather*}
        p_i = a_1^i p_1 + \cdots + a_{d+1}^i p_{d+1},\ a_1^i + \cdots + a_{d+1}^i = 1,\\
        q_i = b_1^i q_1 + \cdots + b_{d+1}^i q_{d+1},\ b_1^i + \cdots + b_{d+1}^i = 1,\\
        i = 1, \cdots, n,
    \end{gather*}
    then by the assumption that $A$ and $B$ have the same volume vector (obviously the same primitive volume vector as well), we have
    \begin{equation*}
        a_j^k = b_j^k, \quad k = 1, \cdots, n, \quad j = 1, \cdots, d+1.
    \end{equation*}
    Since unimodular transformations preserve affine combinations, we know that $p_iU+\mathbf{v}=q_i$ for other $i=4,\cdots,n$, and we are done.
\end{proof}

Now we turn to prove Lemma \ref{lemcompleteinv}.

\begin{proof}[Proof of Lemma \ref{lemcompleteinv}]
    The main idea is to "reduce" the simplices $A$ and $B$. It is always possible to apply suitable unimodular transformations $U_p$ and $U_q$ such that the origin remains fixed, $p'_2\coloneqq p_2U_p$ and $q'_2\coloneqq q_2U_q$ lie on the positive $x$-axis, and $p'_1\coloneqq p_1U_p$ and $q'_1\coloneqq q_1U_q$ lie in the first quadrant (including the $y$-axis) with the smallest possible $x$-coordinate. Let the transformed sets be $A'=\{\mathbf{0},p'_1,p'_2\}$ and $B'=\{\mathbf{0},q'_1,q'_2\}$. We now claim that $p'_1=q'_1$ and $p'_2=q'_2$, implying that $A$ and $B$ are equivalent to the same polygon. Thus, the case $d=2$ is proved.

    Now, we prove that $p'_1=q'_1$ and $p'_2=q'_2$. Since $A$ and $B$ have the same lattice height vector, and lattice height is a unimodular invariant, the height of $p'_1$ to the edge $\mathbf{0}p'_2$ is the same as the height of $q'_1$ to the edge $\mathbf{0}q'_2$, denoted by $d_\perp$. Additionally, since the triangles formed by $A$ and $B$ have the same area, $len(\mathbf{0}p'_2)=len(\mathbf{0}q'_2)$, and therefore $p'_2=q'_2$. Let $p'_2=q'_2\coloneqq(x,0)$, and let the $y$-coordinate of the primitive vector in the direction of $p'_1$ be $y_p$, and the $y$-coordinate of the primitive vector in the direction of $q'_1$ be $y_q$. Then, the lattice height of $p'_2$ to the edge $\mathbf{0}p'_1$ is $x\cdot y_p$, and the lattice height of $q'_2$ to the edge $\mathbf{0}q'_1$ is $x\cdot y_q$. Since these two lattice heights must be equal, we have $y_p=y_q$. Consequently, $p'_1=q'_1$ if and only if the primitive vectors in the directions of $p'_1$ and $p'_2$ have the same $y$-coordinate. Therefore, $p'_1=q'_1$.
\end{proof}

\begin{rem}
    The above discussion shows that the volume vector and lattice height vector are two important and useful invariants for unimodular equivalence. This is not surprising: The volume vector implies the information about the affine relationship between vertices, and the lattice height vector implies the lattice structure of the polytope; At the same time, unimodular transformations can be characterised as a group of special affine transformations, which preserve integer lattice $\mathbb{Z}^d$.
\end{rem}

\end{document}